\documentclass[journal,twoside]{IEEEtran}
\usepackage{amsmath,amsthm}
\usepackage[table]{xcolor}
\usepackage{graphicx}
\usepackage{hyperref}

 \newtheorem{theorem}{Theorem}[section]
 \newtheorem{corollary}[theorem]{Corollary}
 \newtheorem{lemma}[theorem]{Lemma}

\def\vc#1{{\bar #1}}

\begin{document}
\title{On $q$-ary shortened-$1$-perfect-like codes%
\thanks{This manuscript is the accepted version
of the paper published in IEEE Trans. Inf. Theory 68(11) 2022, 7100--7106, 
\url{https://doi.org/10.1109/TIT.2022.3187004} \copyright\,IEEE\,2002.}
\thanks{The work of M.\,J.\,Shi is supported by the National Natural Science
Foundation of China (61672036); the work
of R.\,S.\,Wu is supported by China Postdoctoral Science Foundation funded
project under Grant No. 2021M703098; the work of
D.\,S.\,Krotov is supported with the state contract of the
Sobolev Institute of Mathematics (FWNF-2022-0017).
}
} 

\author{%
\href{https://orcid.org/0000-0002-4990-6271}{Minjia Shi}%
\thanks{%
     M. J. Shi is with Key Laboratory of Intelligent Computing and Signal
Processing of Ministry of Education, School of Mathematics Sciences, Anhui
University, Hefei 230601, Anhui, China.
     \protect\\
     E-mail: \href{mailto:smjwcl.good@163.com}{smjwcl.good@163.com}%
}%
, \href{https://orcid.org/0000-0002-2361-5126}{Rongsheng~Wu}%
\thanks{%
     R. S. Wu is with School of Mathematical Sciences, University of Science
and Technology of China, Hefei 230026, Anhui, China.
     \protect\\
     E-mail: \href{mailto:wrs2510@163.com}{wrs2510@ustc.edu.cn}%
}%
, and \href{https://orcid.org/0000-0002-8516-755X}{Denis~S.~Krotov}
\thanks{%
     D. S. Krotov is with Sobolev Institute of Mathematics,
     pr. Akademika Koptyuga 4,
     Novosibirsk 630090, Russia.
     \protect\\
     E-mail: \href{mailto:krotov@math.nsc.ru}{krotov@math.nsc.ru}%
}
} 

\markboth{IEEE Transactions on Information Theory, 2022}%
{Sh, Wu, and Krotov: On $q$-ary shortened-$1$-perfect-like codes}

\date{}
\maketitle

\begin{abstract}
\boldmath
We study codes with parameters of $q$-ary shortened Hamming codes, i.e., $(n=(q^m-q)/(q-1), q^{n-m}, 3)_q$. Firstly, we prove the fact mentioned in 1998 by Brouwer et al. that such codes are optimal, generalizing it to a bound for multifold packings of radius-$1$ balls, with a corollary for multiple coverings. In particular, we show that the punctured Hamming code is an optimal $q$-fold packing with minimum distance $2$. Secondly, for every admissible length starting from $n=20$, we show the existence of $4$-ary codes with parameters of shortened $1$-perfect codes that cannot be obtained by shortening a $1$-perfect code.
\end{abstract}




\section{Introduction}

In this paper, we study unrestricted (in general, nonlinear) codes with parameters of a shortened
Hamming ($1$-perfect) code over a non-binary finite field. Two questions are considered.
Firstly, we provide a generalized proof that these codes are optimal,
which is applicable to more general classes of objects such as multifold packings
and multiple coverings.
The second question is the following: is every code with parameters of a shortened
$1$-perfect code indeed a shortened $1$-perfect code? The answer is known to be positive
in the binary case~\cite{Bla99}.
We show that this is not the case
for quaternary codes
and discuss
some other alphabet sizes.
In preceding works, the questions under study were well
developed for binary codes, and the next two paragraphs present a brief survey on the topic.

By explicit evaluation of the linear-programming bound, Best and Brouwer~\cite{BesBro77}
proved that shortened, doubly-shortened, and triply-shortened
binary Hamming codes are optimal.
Brouwer at al.
claimed in~\cite[Proposition~4.5]{BHOS:1998:2ary3ary}
that the same is true for non-binary shortened Hamming codes.
Etzion and Vardy~\cite{EV:98} asked the following question: given a binary code $C$ of
length $n=2^m-2$ with $|C| = 2^{n-m}$ and minimum Hamming distance $3$, is it always possible to
represent $C$ as a shortened $1$-perfect code?
The same question was asked for shortening more than once.
Blackmore~\cite{Bla99} solved the first half of this problem,
showing that every code with parameters of the binary shortened Hamming code
is indeed a shortened $1$-perfect code.
\"Osterg{\aa}rd and Pottonen~\cite{OstPot:13-512-3} found two
$(13,512,3)$ binary codes which are not doubly-shortened $1$-perfect codes.
A generalization on these codes to an arbitrary length of form $2^m-3$
(and also $2^m-4$ for the triply-shortened analog), $m\ge 5$,
was constructed in~\cite{KOP:2011} (where
also all $(12,256,3)$ and $(13,512,3)$
codes were classified up to equivalence).
However, it was shown in~\cite{Kro:2m-3} and~\cite{Kro:2m-4}
that every code with parameters of doubly- or
triply-shortened binary Hamming code induces a very regular structure
called equitable partition (in the doubly-shortened case,
the code is a cell of such partition,
while in the triply-shortened extending the code results in a cell of such partition).
Theoretical and computational results in
\cite{KKO:smallMDS}
confirm that for $q\le 8$, every $(q,q^{q-2},3)_q$ code is
a shortened $1$-perfect
MDS $(q+1,q^{q-1},3)_q$ code (see also the discussion in the end of Section~\ref{s:5555}).

Generalizing Best and Brouwer's bound~\cite{BesBro77},
Krotov and Potapov~\cite[Theorem~4]{KroPot:multifold}
extended it to the case of multifold
packing of radius-$1$ balls in binary Hamming spaces.
Multifold, or $\lambda$-fold, packings
are a combinatorial treatment of the codes called list-decodable
(see e.g. the surveys in~\cite{Guruswami:05} and~\cite{Resch:20}), the case $\lambda=1$
corresponding to ordinary error-correcting codes. As a corollary, the upper bounds
on the size of multifold packing of radius-$1$ balls obtained in~\cite{KroPot:multifold}
were treated there (Theorem~6) as lower bounds on the size of multiple $1$-coverings, which can be considered as the complements of multifold packings.

The purpose of this paper is two-fold:
one is to derive an upper bound and a lower bound
on the size of $q$-ary $\lambda$-fold
$1$-packing and multiple radius-$1$ covering in $H(n,q)$ with $n\equiv q \bmod q^2$, respectively;
and the other is to characterize codes with parameters of $q$-ary shortened Hamming codes.

The structure of the paper is as follows.
Section~\ref{s:def} contains definitions and notations.
In Section~\ref{s:multi}, we prove
the first main result, which
partially generalizes the bound~\cite[Theorem~4]{KroPot:multifold}
to the non-binary case, as well as it generalizes the optimality
of the $1$-shortened Hamming $q$-ary code~\cite[Proposition~4.5]{BHOS:1998:2ary3ary}
to an upper bound for the size of a multifold packing.
Here, ``partially'' means that we are only able to generalize case~(c)
in~\cite[Theorem~4]{KroPot:multifold},
corresponding to lengths $n$ congruent to $q$ modulo $q^2$
(apart from case~(d), which is just a sphere-packing bound);
in particular, it remains unknown for $q>2$
if all $q$-ary doubly (or-more-times) shortened Hamming codes are optimal or not.
We also treat the result as a lower bound on the size of a multiple covering (Section~\ref{s:cov}).
Section~\ref{s:main2} contains the second main result of the paper, Theorem~\ref{th:4}, where
we show the existence of $4$-ary codes with parameters
$(n=(4^m-4)/3, 4^{n-m}, 3)_4$ (of a shortened $1$-perfect code)
that cannot be obtained by shortening a $1$-perfect code;
the similar problem for
alphabets of other sizes
is briefly discussed in Section~\ref{s:5555}. Section~\ref{s:end} contains concluding remarks.

\section{Preliminaries}\label{s:def}

The \emph{Hamming graph} $H(n,q)$
is a graph on the set of $n$-words over the alphabet
$\{0,1,\ldots,q-1\}$, two words being adjacent if they differ
in exactly one position.
The minimum-path distance in $H(n,q)$ is called the \emph{Hamming
distance} and denoted by
$d(\cdot,\cdot)$.

A $q$-ary \emph{code} of length $n$ is an arbitrary nonempty
set of vertices of $H(n,q)$.
In particular, a $q$-ary code of length $n$,
size $M$ and minimum distance $d$ between codewords
is referred to as
an $(n,M,d)_q$-code.
The numbers $q$, $n$, $M$, and $d$ are called the
\emph{parameters} of the code.

For any two words or symbols $\vc{x}$ and $\vc{y}$,
by $\vc{x}\vc{y}$ we denote their concatenation.
For a code $C$ and a symbol or word $\vc{x}$, we denote
$C\vc{x}:=\{\vc{c}\vc{x}:\ \vc{c}\in C\}$
and
$\vc{x}C:=\{\vc{x}\vc{c}:\ \vc{c}\in C\}$;
similarly,
$CD:=\{\vc{x}\vc{y}:\ \vc{x}\in C,\ \vc{y}\in D\}$
for two codes $C$ and $D$.

 For a code $C$ in $H(n,q)$,
 denote by $C^{(i)}$ the set of vertices
 at distance $i$ from $C$.
 A code $C$ is called \emph{completely regular}
 if for any $i$ and $j$ and
 the number of neighbors in $C^{(j)}$ is the same
 for all words of $C^{(i)}$ (see~\cite{BRZ:CR} for a recent survey).
 The last concept does not play
 a role in the present research,
 but we find important to note
 (see Corollary~\ref{c:cr}) that the codes we consider
 are possessed of such exceptional regularity properties.

Let $C$ be an $(n,M,d)_q$-code. Assume that $j\in \{1,2,\ldots,n\}$ and let $C(j)$ be the code obtained from $C$ by selecting only these words of $C$ having $\alpha\in \{0,1,\ldots,q-1\}$ in position $j$. Puncturing the set $C(j)$ on $j$ gives a code denoted $C^-$.
Then we say that $C^-$ is a \emph{shortened} (with $\alpha$ in position $j$) $C$
and that $C$ is a \emph{lengthened} $C^-$.
It is easy to see that the minimum distance of the code $C^-$ is at least as great as that of $C$.

A \emph{$1$-perfect code} in $H(n,q)$ (or any other graph) is a set $C$ of vertices such that every vertex of $H(n,q)$ is at distance $0$ or $1$ from exactly one element of $C$.
In this paper, we focus on the codes with parameters
of shortened $1$-perfect codes, called
\emph{shortened-$1$-perfect-like}
(as is stated in one of our main results not all such codes
are indeed shortened $1$-perfect).
If $q$ is a prime power, then $1$-perfect codes exist if and only if
$n=(q^m-1)/(q-1)$, $m\in\{1,2,\ldots\}$;
in particular, there is a linear such code
known as the \emph{Hamming code}
(\emph{linear} means that the vertex set is endowed with the structure of a vector space over the finite field of order $q$ and the code is a vector subspace).
If $q$ is not a prime power, no $1$-perfect codes are known;
however, most of our results, unless otherwise is stated, have no restrictions on $q$.

A \emph{$\lambda$-fold $r$-packing} in $H(n,q)$ is a multiset $C$ of vertices such that the radius-$r$ balls centered in the words of $C$ cover each vertex of $H(n,q)$ by not more than $\lambda$ times.
A set $C$ of vertices in $H(n,q)$ is called a $q$-ary $(n,\cdot ,r,\mu)$ multiple covering if for every vertex $\vc{x}$ in $H(n,q)$ the number of elements of $C$ at distance at most $r$ from $\vc{x}$
is not less than $\mu$.

\section{Bounds for multifold packings and multiple coverings}\label{s:multi}

\subsection{Multifold packings}\label{s:pack}
Before we state the main results of this section
(Theorems~\ref{th:ub} and~\ref{th:dist2}),
we define additional concepts and notations, used in the proofs.

Let $C$ be a code in $H(n,q)$.
The sequence $(A_i)_{i=0}^n$, where
\begin{IEEEeqnarray*}{rCl}
A_i&:=&\frac{1}{|C|}\sum_{\vc{x}\in C}A_i(\vc{x}),
\\
A_i(\vc{x})&:=&|\{\vc{y}\in C:\ d(\vc{x},\vc{y})=i\}|,
\end{IEEEeqnarray*}
is called the
\emph{distance distribution} of $C$
(while $(A_i(\vc{x}))_{i=0}^n$
is the \emph{weight distribution} of the code $C-\vc{x}$).

Let $n$ be a positive integer and $i$ an indeterminate.
The Krawtchouk polynomial (see, e.g., \cite[Ch.\,5 \S7]{MWS})
$K_k(i)$ is defined to be
$$K_k(i):=\sum_{j=0}^k(-1)^j(q-1)^{k-j}\binom{i}{j}\binom{n-i}{k-j}, \ k=0,1,2, \ldots.$$
It is worth noting that the Krawtchouk polynomial plays a pivotal role in this section.
Based on that, we define the dual sequence $\{B_k\}_{k=0}^n$ by
$$B_k:=\frac{1}{|C|}\sum_{i=0}^nA_iK_k(i).$$
In particular, we have (\cite[Theorem~5.20?]{MWS})
$$
q^nA_k=|C|\sum_{i=0}^nB_iK_k(i), \ \ \ \ \ k\in \{0,1,\ldots ,n\}.
$$
We give a detailed calculation for $K_0(i)$, $K_1(i)$, and $K_2(i)$ as follows:
\begin{IEEEeqnarray*}{rCl}
  K_0(i) &=& 1;\\
  K_1(i) &=& n(q-1)-qi;\\
  K_2(i) &=& \frac{1}{2}\big(n(q-1)-qi\big)^2-\frac{n(q-1)^2}{2}+\frac{q(q-2)}{2}i.
\end{IEEEeqnarray*}

\begin{lemma}\label{l:A012}
For a $\lambda$-fold $1$-packing $C$ in $H(n,q)$, $q>2$, with distance distribution $(A_0,A_1, \ldots, A_n)$, it holds
\begin{equation}\label{eq:odd}
 n(q-1)A_0 + 2(q-1)A_1+2A_2
\le
(n+1)(q-1)\lambda - q+1.
\end{equation}
Moreover, if $q$, $n$, and $\lambda$ are even, then
\begin{equation}\label{eq:even}
n(q-1)A_0 + 2(q-1)A_1+2A_2
\le
(n+1)(q-1)\lambda-q.
\end{equation}
In particular, if one of the relations above holds with equality,
then we have $A_0=1$ (that is, there are no codewords of $C$ with the multiplicity more than one)
and $A_1=\lambda-1$.
\end{lemma}
\begin{proof}
By the definition of $A_i$, it is sufficient to prove the corresponding
inequalities for $A_i(\vc{x})$, for any $\vc{x}$ from $C$.
Suppose $\vc{x}\in C$ and denote
$$S:=\{\vc{y}\in H(n,q):\ d(\vc{x},\vc{y})=1\}.$$
The codeword $\vc{x}$ is neighbor to $n(q-1)$ words $\vc{y}$ of $S$;
every codeword of $S$ is within radius $1$ from exactly $q-1$ words $\vc{y}$ of $S$;
every codeword at distance $2$ from $\vc{x}$
is neighbor to exactly $2$ words $\vc{y}$ of $S$.
Since, by the definition of a $\lambda$-fold $1$-packing, each of
$n(q-1)$ elements of $S$ are within radius $1$ from at most $\lambda$
codewords, we get
\begin{multline*}
  n(q-1)A_0(\vc{x})+(q-1)A_1(\vc{x})+2A_2(\vc{x}) \\ = \sum_{\vc{y}\in S}\left|\{\vc{c}\in C:\ d(\vc{c},\vc{y})\le 1\}\right|
   \le n(q-1)\lambda.
\end{multline*}

Next, for $\vc{x}\in C$ we see that
\begin{multline*}
(q-1)(A_0(\vc{x})+A_1(\vc{x}))
\\ =(q-1)\cdot |\{\vc{c}\in C:\ d(\vc{c},\vc{x})\le 1   \}|
\le (q-1)\lambda.
\end{multline*}
We conclude that
\begin{multline}\label{eq:chain}
n(q-1)A_0(\vc{x})+2(q-1)A_1(\vc{x})+2A_2(\vc{x})\\
 =
\big(n(q-1)A_0(\vc{x})+(q-1)A_1(\vc{x})+2A_2(\vc{x})\big)\\
+(q-1)(A_0(\vc{x})+A_1(\vc{x}))-(q-1)A_0(\vc{x}) \\
  \le
n(q-1)\lambda + (q-1)\lambda - (q-1)A_0(\vc{x})
\\
\le
(n+1)(q-1)\lambda-(q-1),
\end{multline}
where the  last inequality comes from $A_0(\vc{x})\geq 1$.
Moreover, if $q$, $n$, and $\lambda$ are even,
then the left part of~\eqref{eq:chain} is even, while the right is
odd; so, the equality is impossible in this case and we have \eqref{eq:even}.

It remains to note that if $A_0(\vc{x})>1$ or $A_1(\vc{x}) < \lambda-1$ for $\vc{x}\in C$,
then the right part of~\eqref{eq:chain} decreases by at least $q-1\ge 2$,
which makes impossible the equality in~\eqref{eq:even} or~\eqref{eq:odd}.
\end{proof}

\begin{corollary}
The statement of Lemma~\ref{l:A012} remains true if
the code contains the all-zero word and the distance distribution is replaced by
the weight distribution in the inequality.
\end{corollary}

\begin{theorem}\label{th:ub}
If $q>2$ and $n\equiv q \bmod q^2$, then every $q$-ary $\lambda$-fold $1$-packing $C$ of length $n$ satisfies
$$|C| \le \frac{q^n\big((n+1)\lambda-1\big)}{n^2(q-1)+nq}.$$
Moreover, if $q$, $n$, and $\lambda$ are even, then a slightly improved bound gives
$$|C| \le \frac{q^n \big((n+1)(q-1)\lambda-q\big)}{n(q-1)\big(n(q-1)+q\big)}.$$
In each cases, if the equality holds in the bound,
then  $C$ has no multiple codewords.
\end{theorem}
\begin{proof}
Let $n\equiv q \bmod q^2$. Define
\begin{IEEEeqnarray*}{rCl}
  \alpha(i) &=& \big(n(q-1)-qi\big)\big(n(q-1)-qi+q\big) \\
   &=& n(q-1)K_0(i)+2(q-1)K_1(i)+2K_2(i).
\end{IEEEeqnarray*}
It is easy to check that the values of $\alpha(i)$
are nonnegative and vanish only for $i \in \big\{\frac{n(q-1)}{q}, \frac{n(q-1)}{q}+1\big\}$. Hence, we have
\begin{IEEEeqnarray*}{rCl}
  \alpha(0)B_0 &\leq & \sum_{i=0}^n \alpha(i)B_i \\
   &=& \frac{q^n}{|C|}\big(n(q-1)A_0+2(q-1)A_1+2A_2\big)\\
   &\leq&  \frac{q^n}{|C|}\big((n+1)(q-1)\lambda-q+1\big),
\end{IEEEeqnarray*}
where the last inequality comes
from~\eqref{eq:odd}
in~Lemma~\ref{l:A012}. Since $\alpha(0)=n(q-1)(n(q-1)+q)$, we have
\begin{IEEEeqnarray*}{rCl}
  |C| &\leq& q^n \cdot \frac{(n+1)(q-1)\lambda-q+1}{n(q-1)(n(q-1)+q)}=\frac{q^n((n+1)\lambda-1)}{n^2(q-1)+nq},
\end{IEEEeqnarray*}
which is the first inequality in the claim of the theorem. The second inequality is proved
similarly, using~\eqref{eq:even}
in~Lemma~\ref{l:A012}. Also by Lemma~\ref{l:A012}, the equality implies that
$C$ has no multiple codewords.
\end{proof}

The \emph{uniformly packed codes},
in the sense of~\cite{GoeTil:UPC},
are special cases of completely regular codes
and mentioned below without a definition.
\begin{corollary}\label{c:cr}
 All shortened-$1$-perfect-like codes are completely regular and uniformly packed.
\end{corollary}
\begin{proof}
 As follows from the proof of Theorem~\ref{th:ub},
 a multifold packing attaining the bound
 (in particular, a shortened-$1$-perfect-like code)
 has at most two nonzero
 coefficients $B_i$, apart from $B_0$.
 Since the minimum distance of a
 shortened-$1$-perfect-like code is $3$,
 it is uniformly packed by~\cite[Theorem~12]{GoeTil:UPC}
 and completely regular by \cite[Theorem~5.13]{Delsarte:1973}
 (note that our definition of completely regular codes differs from the original one in~\cite{Delsarte:1973},
 but is equivalent to it due to~\cite{Neumaier92}).
\end{proof}

With an additional restriction on the minimum distance,
the bound in Theorem~\ref{th:ub} can be improved as follows.

\begin{theorem}\label{th:dist2}
If $q>2$ and $n\equiv q \bmod q^2$,
then every $\lambda$-fold $1$-packing $C$
in $H(n,q)$
with minimum distance $2$ satisfies
\begin{equation}\label{eq:dist2}
|C| \le \frac{\lambda q^n}{n(q-1)+q}.
\end{equation}
\end{theorem}
\begin{proof}
Let $C$ be a $\lambda$-fold $1$-packing with distance distribution $(A_0,A_1, \ldots , A_n)$. Then for any $\bar{x}\in C$, we have
\begin{multline*}
  A_2(\bar{x})=|\{\bar{z}\in C: d(\bar{x},\bar{z})=2\}|
  \\
  \le \bigg\lfloor \frac{n(q-1)(\lambda-A_0(\bar{x}))}{2}\bigg\rfloor
   = \frac{n(q-1)(\lambda-A_0(\bar{x}))}{2},
\end{multline*}
where the last equality comes from $n\equiv q \bmod q^2$.
Using $A_1(\bar{x})=0$ yields
\begin{multline}\label{eq:chain1}
n(q-1)A_0(\bar{x})+2(q-1)A_1(\bar{x})+2A_2(\bar{x})\\
 \le
n(q-1)A_0(\bar{x})+n(q-1)(\lambda-A_0(\bar{x}))
=
n(q-1)\lambda.
\end{multline}
Then an argument similar to that in the proof of Theorem \ref{th:ub} gives the upper bound for the size of $C$.
\end{proof}

\begin{corollary}\label{cor:punct}
If $q$ is a prime power and $\lambda\in\{1,\ldots,q\}$,
then the maximum size of a $\lambda$-fold $1$-packing
with minimum distance at least $2$ in $H(n,q)$,
where $n=(q^m-q)/(q-1)$,
$m=2,3,\ldots$,
is exactly $\lambda q^n / (nq-n+q)$.
\end{corollary}
\begin{proof}
 The upper bound is claimed in Theorem~\ref{th:dist2}.
 To establish the lower bound, we can easily construct
 a $\lambda$-fold $1$-packing of size
 $\lambda q^n / (nq-n+q)$
 as the union of $\lambda$ cosets of the shortened
 Hamming $((q^m-q)/(q-1),q^{n-m},3)_q$ code.
 To guarantee the minimum distance at least $2$ between
 codewords, we take only cosets that are included
 in the punctured Hamming
 $((q^m-q)/(q-1),q^{n-m+1},2)_q$ code.
\end{proof}

\subsection{Multiple coverings}\label{s:cov}

The following lower bound for the size of $q$-ary $(n,\cdot,1,\mu)$ multiple covering is derived from Theorem~\ref{th:ub}. It should be noted that,
in contrast to multifold packing, multiple coverings are defined as ordinary sets,
and the complement argument below does not work for their multiset analogs.

\begin{corollary}\label{th:lb}
If $q>2$ and $n\equiv q \bmod q^2$, then every $q$-ary $(n,\cdot,1,\mu)$ multiple covering $\overline C$ of length $n$ satisfies
$$|\overline C| \geq \frac{q^n(n+1)\mu}{n^2(q-1)+nq}.$$
\end{corollary}
\begin{proof}
Let $C$ be a $\lambda$-fold $1$-packing in $H(n,q)$. Then it is easy to check its complement $\overline C=H(n,q)\backslash C$ is a $q$-ary $(n,\cdot,1,\mu)$ multiple covering, where $\mu=n(q-1)+1-\lambda$. According to Theorem \ref{th:ub}, we have
\begin{IEEEeqnarray*}{rCl}
  q^n-|\overline C| &\le& \frac{q^n\Big((n+1)(n(q-1)+1-\mu)-1\Big)}{n^2(q-1)+nq}.
\end{IEEEeqnarray*}
This produces our bound.
\end{proof}

\section[Shortened-1-perfect-like codes that are not shortened 1-perfect]{Shortened-$1$-perfect-like codes that are not shortened $1$-perfect}\label{s:main2}

The main goal of this section is to construct codes with parameters of shortened $1$-perfect codes,
\emph{shortened-$1$-perfect-like codes},
that cannot be obtained by shortening $1$-perfect codes.
We construct such codes for the alphabet of size $4$.
In the last subsection we briefly discuss other alphabets.
The parameters of some codes that occur in our theory
attain the Singleton bound $M\le q^{n-d+1}$; such codes are known as \emph{maximum distance separable} (\emph{MDS}, for short) codes.
\subsection{Concatenation construction}
In this section, we adopt
the concatenation construction~\cite{Romanov:2019}
for $1$-perfect codes (a generalization of the Solov'eva construction~\cite{Sol:81},
whose special case, in its turn, can be treated in terms of earlier concatenation construction of Heden~\cite[Theorem~1]{Heden:77})
to construct shortened-$1$-perfect-like codes. As we will see below, not all codes obtained by this construction
can be obtained by shortening from $1$-perfect codes.

\begin{lemma}[Romanov \cite{Romanov:2019}]\label{l:rom}
Assume
$n=(q^m-1)/(q-1)$, $n'=(q^{m-1}-1)/(q-1)$, $n''=q^{m-1}$.
Let $(C_0, \ldots ,C_{n''-1})$ be a partition
of the Hamming space $H(n',q)$ into
$1$-perfect $(n', q^{n'-(m-1)}, 3)_q$ codes.
Let $(D_0, \ldots ,D_{n''-1})$ be a partition
of an $(n'',q^{n''-1},2)_q$ MDS code into $n''$
  codes with parameters
  $(n'', q^{n''-m}, 3)_q$.
Then the code
\begin{equation}
 \label{eq:P}
P:= \bigcup_{i=0}^{q^{m-1}-1} D_i  C_i
\end{equation}
is a $1$-perfect $(n, q^{n-m}, 3)_q$ code.
\end{lemma}

To adopt the construction to the parameters under study,
we replace the perfect codes $C_i$ by codes with the parameters of
a shortened $1$-perfect code.

\begin{theorem}\label{th:1}
Assume
$n=(q^m-1)/(q-1)$, $n'=(q^{m-1}-1)/(q-1)$, $n''=q^{m-1}$.
Let $(B_0, \ldots ,B_{n''-1})$ be a partition
of the Hamming space $H(n'-1,q)$
into
 $(n'-1, q^{n'-1-(m-1)}, 3)_q$ codes.
Let $(D_0, \ldots ,D_{n''-1})$ be a partition
of the $(n'',q^{n''-1},2)_q$ MDS code $M_0$ into $n''$
  codes with parameters
  $(n'', q^{n''-m}, 3)_q$, where
  \begin{multline*}
  M_a:=\big\{ x_1 \ldots x_{n''} :\  x_1+ \ldots +x_{n''}\equiv a \bmod q\big\},
  \\
  a\in\{0,\ldots,q-1\}.
  \end{multline*}
  \begin{itemize}
   \item[\rm(i)] the code
$$
S:= \bigcup_{i=0}^{q^{m-1}-1} D_i B_i
$$
is an $(n-1, q^{n-1-m}, 3)_q$ code.
\item[\rm(ii)] $S$ can be lengthened
to a $1$-perfect $(n, q^{n-m}, 3)_q$ code
if and only if every $B_i$
can be lengthened to a $1$-perfect
$(n', q^{n'-(m-1)}, 3)_q$ code $C_i$ such that
$C_0$, \ldots, $C_{q^{m-1}-1}$ form a partition of the Hamming space $H(n',q)$.
\item[\rm(iii)] There is a partition of $H(n-1,q)$ into $q^m$ codes of parameters $(n-1, q^{n-1-m}, 3)_q$,
one of which is $S$.
  \end{itemize}
\end{theorem}
\begin{proof} (i)
The proof of the parameters of $S$ is straightforward:
the cardinality equals
$\sum_{i=0}^{q^{m-1}-1} |D_i|\cdot|B_i|$,
the minimum distance is $3$ because
minimum distance of $D_i B_i$ is $3$
and the distance between $D_i B_i$ and $D_{i'} B_{i'}$, $i\ne i'$ is $3$
(the distance between $D_i$ and $D_{i'}$ is $2$;
the distance between $B_i$ and $B_{i'}$ is $1$).

(ii)
If $B_i$ is the shortened $C_i$, $i=0, \ldots, q^{m-1}-1$,
where $(C_0, \ldots, C_{q^{m-1}-1})$ is a partition of $H(n',q)$.
Then $S$ is the shortened $P$ from \eqref{eq:P}.

It remains to show the inverse,
that the partition
$(B_0, \ldots, B_{q^{m-1}-1})$ can be lengthened if $S$ can.
Assume that $S$ is lengthened to some $1$-perfect code $P$,
i.e., $S=S_0$, where
$$P=\bigcup_{a=0}^{q-1}S_a a.$$
From each $D_i$, we can choose a representative $\vc{d}_i$
such that $d(\vc{d}_i,\vc{d}_{i'})=2$ if $i\ne i'$
(indeed, if we take a word $\vc{x}$ of length $n''$
not from $\bigcup_{i=0}^{n''-1} D_i$,
then by properties of a distance-$2$
MDS code it has $n''$ neighbors in
$\bigcup_{i=0}^{n''-1} D_i$,
one neighbor in each direction,
and hence one neighbor, $\vc{d}_i$, in each $D_i$ because the minimum distance of $D_i$ is $3$).

Now, consider an arbitrary word $\vc{x}$ from $B_i^{(2)}$.
The word $\vc{d}_i \vc{x}$ has no codeword neighbors in $S$,
because changing $\vc{d}_i$ in one position sends it out of $M_0$,
while changing in one position $\vc{x}$ sends it to $B_i^{(1)}$
but not to $B_i$ (we note that $B_i^{(3)}$ is empty because
$B_i$ is an optimal distance-$3$ code).
Therefore,
$\vc{d}_i \vc{x} 0 $ has no neighbors in $S0$.
Since $P$ is $1$-perfect, $\vc{d}_i \vc{x} 0 $
has a neighbor $\vc{d}_i \vc{x} a$ from $P$
for some $a \in \{1, \ldots,q-1\}$.
It follows that
$$B_i^{(2)} = \bigcup_{a=1}^{q-1} B_i^{a}, \qquad \mbox{where }
B_i^{a}:=\{ \vc{x} \in B_i^{(2)}: \ \vc{d}_i \vc{x} a \in P \}.
$$
Denote
$$ C_i:= B_i 0 \cup  \bigcup_{a=1}^{q-1} B_i^{a} a, \qquad i \in \{0, \ldots, q^{m-1}-1 \}. $$
Since $\vc{d}_i C_i$ is a subset of $P$, we see that $C_i$ is a distance-$3$ code.
On the other hand,
the cardinality of this code
is
\begin{multline*}
 |C_i| = |B_i \cup B_i^{(2)}|
= q^{n'-1} - |B_i^{(1)}|
\\
= q^{n'-1}
- (q-1)(n'-1)
\cdot|B_i|
\\
= q^{n'-1} -
(q^{m-1}-q)
\cdot q^{n'-m}
=
q^{n'-m+1}.
\end{multline*}
That is, $C_i$ is
a $1$-perfect
$(n',q^{n'-m+1},3)_q$ code;
moreover, by construction,
$B_i$ is the shortened $C_i$.
It remains to note that
$C_i$ and $C_j$, $i\ne j$,
are disjoint
because the distance between
$\vc{d}_i C_i$ and $\vc{d}_j C_j$,
which are both subsets of $P$,
is at least $3$.

(iii)
To construct a partition,
we need to construct $q^m$ disjoint
codes with the same parameters as $S$.
Cyclically permuting the indices of the codes
$D_0$, \ldots, $D_{q^{m-1}-1}$ in (i),
we obtain $q^{m-1}$ different codes
with the same parameters as $S$,
and it is easy to see that
they are pairwise disjoint
(indeed, if $D_iB_{i'}$ intersects with $D_jB_{j'}$,
then $i=j$ and $i'=j'$).
The codewords of all these
codes satisfiy the following
property: the subword consisting of the first
$n''$ symbols of the codeword belongs to $M_0$.
Adding modulo $q$
some element $a\in\{0,1, \ldots,q-1\}$
to the first symbol of all codewords
changes this property from $M_0$ to $M_a$;
so,
we have $q^{m-1}$ disjoint codes for each $a$,
and codes with different $a$ are also disjoint.
This gives a required partition.
\end{proof}

\subsection{Non-lengthenable partitions}\label{s:4444}
Here, for $q=4$,
we will present a partition
$(B_0, \ldots, B_{q^2-1})$
of $H(q,q)$ into
$(q,q^{q-2},3)_q$ codes
that cannot be obtained
by shortening
$(q+1,q^{q-1},3)_q$ codes that form
a partition of $H(q+1,q)$.
Note that in the notation of Theorem~\ref{th:1},
these codes correspond to $m=3$;
moreover, for this special case,
$(q,q^{q-2},3)_q$ and
$(q+1,q^{q-1},3)_q$ are parameters of distance-$3$ MDS codes.

For $q=4$, a $(q,q^{q-2},3)_q$ code $B$ is unique up to isomorphism
(indeed such a code is the graph $\{xyf(x,y)g(x,y) :\ x,y\in\{0,1,2,3\}\}$
of the pair $f$, $g$ of orthogonal latin squares of order $4$).
Moreover, $B^{(2)}$ is uniquely partitioned into three distance-$3$ codes.
However, there are a lot of non-isomorphic partitions of $H(4,4)$ into
$(4,4^2,3)_4$
codes. Below, we draw one of them (we denote $\mathrm{A}:=10$, \ldots, $\mathrm{F}:=15$).
$$
\def\A{\mathrm{A}}
\def\B{\mathrm{B}}
\def\C{\mathrm{C}}
\def\D{\mathrm{D}}
\def\E{\mathrm{E}}
\def\F{\mathrm{F}}
\large\begin{array}{|c@{\,}c@{\,}c@{\,}c|c@{\,}c@{\,}c@{\,}c|c@{\,}c@{\,}c@{\,}c|c@{\,}c@{\,}c@{\,}c|}
 \hline
  0&  1&  2&  3&    4&  5&  6&  7&    8&  9& \A& \B&   \C& \D& \E& \F \\[-0.2EM]
 \F& \E&  7&  6&    1&  0&  3& \A&    2&  4& \D& \C&    9& \B&  8&  5 \\[-0.2EM]
 \A&  8& \B&  9&   \E& \F& \C& \D&    5&  6&  0&  1&    3&  7&  4&  2 \\[-0.2EM]
 \D& \C&  5&  4&   \B&  2&  9&  8&    7&  3& \F& \E&    6& \A&  1&  0 \\
  \hline
  6&  7& \F& \E&   \A&  3&  0&  1&   \D& \C&  4&  5&    2&  8& \B&  9 \\[-0.2EM]
  3&  5&  1&  0&    7&  6&  2&  4&   \B& \A&  9&  8&   \E& \F& \C& \D \\[-0.2EM]
  4&  2& \D& \C&    8&  9&  5& \B&   \F& \E&  3&  7&    0&  1& \A&  6 \\[-0.2EM]
  9& \B&  8& \A&   \C& \D& \E& \F&    1&  0&  6&  2&    5&  4&  7&  3 \\
  \hline
 \B& \A&  9&  8&   \F& \E& \D& \C&    3&  2&  7&  0&    1&  6&  5&  4 \\[-0.2EM]
 \C& \D&  4&  2&    5&  8& \B&  9&    6&  1& \E& \F&   \A&  3&  0&  7 \\[-0.2EM]
  7&  0&  6&  5&    2&  4&  1&  3&    9& \B&  8& \A&   \D& \C& \F& \E \\[-0.2EM]
 \E& \F&  3&  1&    0&  7& \A&  6&    4&  5& \C& \D&    8&  9&  2& \B \\
  \hline
  5&  4& \C& \D&    9& \B&  8&  2&   \E& \F&  1&  6&    7&  0&  3& \A \\[-0.2EM]
  8&  9& \A& \B&   \D& \C& \F& \E&    0&  7&  5&  3&    4&  2&  6&  1 \\[-0.2EM]
  1&  3& \E& \F&    6& \A&  7&  0&   \C& \D&  2&  4&   \B&  5&  9&  8 \\[-0.2EM]
  2&  6&  0&  7&    3&  1&  4&  5&   \A&  8& \B&  9&   \F& \E& \D& \C \\
  \hline
\end{array}
$$
The codes $B_2^1$, $B_2^2$, $B_2^3$, up to enumeration, are as follows, where
$\{\alpha,\beta,\gamma\} =\{1,2,3\}$:
\begin{center}
\scalebox{1.2}{
$
\def\0{2}
\def\A{\alpha}
\def\C{\gamma}
\def\B{\beta}
\def\X{\cdot}
\begin{array}{|c@{\,}c@{\,}c@{\,}c|c@{\,}c@{\,}c@{\,}c|c@{\,}c@{\,}c@{\,}c|c@{\,}c@{\,}c@{\,}c|}
 \hline
\X & \X & \0 & \X &  \C & \X & \X & \X &  \X & \X & \X & \B &  \X & \cellcolor{green!20}\A & \X & \X \\[-0.5EM]
\X & \B & \X & \X &  \X & \X & \X & \A &  \0 & \X & \X & \X &  \X & \X & \C & \X \\[-0.5EM]
\cellcolor{blue!15}\A & \X & \X & \X &  \X & \X & \B & \X &  \X & \C & \X & \X &  \X & \X & \X & \0 \\[-0.5EM]
\X & \X & \X & \C &  \X & \0 & \X & \X &  \X & \X & \A & \X &  \B & \X & \X & \X \\\hline

\X & \X & \X & \A &  \X & \B & \X & \X &  \X & \X & \C & \X &  \0 & \X & \X & \X \\[-0.5EM]
\C & \X & \X & \X &  \X & \X & \0 & \X &  \X & \cellcolor{green!20}\A & \X & \X &  \X & \X & \X & \B \\[-0.5EM]
\X & \0 & \X & \X &  \X & \X & \X & \C &  \B & \X & \X & \X &  \X & \X & \A & \X \\[-0.5EM]
\X & \X & \B & \X &  \A & \X & \X & \X &  \X & \X & \X & \0 &  \X & \C & \X & \X \\\hline

\B & \X & \X & \X &  \X & \X & \A & \X &  \X & \0 & \X & \X &  \X & \X & \X & \C \\[-0.5EM]
\X & \X & \X & \0 &  \X & \C & \X & \X &  \X & \X & \B & \X &  \A & \X & \X & \X \\[-0.5EM]
\X & \X & \C & \X &  \0 & \X & \X & \X &  \X & \X & \X & \A &  \X & \B & \X & \X \\[-0.5EM]
\X & \A & \X & \X &  \X & \X & \X & \B &  \C & \X & \X & \X &  \X & \X & \0 & \X \\\hline

\X & \C & \X & \X &  \X & \X & \X & \0 &  \A & \X & \X & \X &  \X & \X & \B & \X \\[-0.5EM]
\X & \X & \A & \X &  \B & \X & \X & \X &  \X & \X & \X & \C &  \X & \0 & \X & \X \\[-0.5EM]
\X & \X & \X & \B &  \X & \A & \X & \X &  \X & \X & \0 & \X &  \C & \X & \X & \X \\[-0.5EM]
\0 & \X & \X & \X &  \X & \X & \C & \X &  \X & \B & \X & \X &  \X & \X & \X & \A \\\hline
\end{array}
$
}
\end{center}

Similarly, $B_3^1$, $B_3^2$, $B_3^3$:
\begin{center}
\scalebox{1.2}{
$
\def\0{3}
\def\A{\tilde \alpha}
\def\C{\tilde\gamma}
\def\B{\tilde\beta}
\def\X{\cdot}
\begin{array}{|c@{\,}c@{\,}c@{\,}c|c@{\,}c@{\,}c@{\,}c|c@{\,}c@{\,}c@{\,}c|c@{\,}c@{\,}c@{\,}c|}
 \hline
\X & \X & \X & \0 &  \A & \X & \X & \X &  \X & \X & \C & \X &  \X & \cellcolor{green!20}\B & \X & \X \\[-0.5EM]
\X & \A & \X & \X &  \X & \X & \0 & \X &  \B & \X & \X & \X &  \X & \X & \X & \C \\[-0.5EM]
\X & \X & \B & \X &  \X & \C & \X & \X &  \X & \X & \X & \A &  \0 & \X & \X & \X \\[-0.5EM]
\C & \X & \X & \X &  \X & \X & \X & \B &  \X & \0 & \X & \X &  \X & \X & \A & \X \\\hline

\X & \X & \A & \X &  \X & \0 & \X & \X &  \X & \X & \X & \B &  \C & \X & \X & \X \\[-0.5EM]
\0 & \X & \X & \X &  \X & \X & \X & \cellcolor{red!20}\A &  \X & \cellcolor{green!20}\C & \X & \X &  \X & \X & \B & \X \\[-0.5EM]
\X & \X & \X & \C &  \B & \X & \X & \X &  \X & \X & \0 & \X &  \X & \A & \X & \X \\[-0.5EM]
\X & \B & \X & \X &  \X & \X & \C & \X &  \A & \X & \X & \X &  \X & \X & \X & \X \\\hline

\X & \C & \X & \X &  \X & \X & \B & \X &  \0 & \X & \X & \X &  \X & \X & \X & \A \\[-0.5EM]
\X & \X & \X & \B &  \C & \X & \X & \X &  \X & \X & \A & \X &  \X & \0 & \X & \X \\[-0.5EM]
\A & \X & \X & \X &  \X & \X & \X & \0 &  \X & \B & \X & \X &  \X & \X & \C & \X \\[-0.5EM]
\X & \X & \0 & \X &  \X & \cellcolor{red!20}\A & \X & \X &  \X & \X & \X & \C &  \B & \X & \X & \X \\\hline

\B & \X & \X & \X &  \X & \X & \X & \C &  \X & \A & \X & \X &  \X & \X & \0 & \X \\[-0.5EM]
\X & \X & \C & \X &  \X & \B & \X & \X &  \X & \X & \X & \0 &  \A & \X & \X & \X \\[-0.5EM]
\X & \0 & \X & \X &  \X & \X & \A & \X &  \C & \X & \X & \X &  \X & \X & \X & \B \\[-0.5EM]
\X & \X & \X & \A &  \0 & \X & \X & \X &  \X & \X & \B & \X &  \X & \C & \X & \X \\\hline
\end{array}
$
}
\end{center}

and $B_{\mathrm{C}}^1$, $B_{\mathrm{C}}^2$, $B_{\mathrm{C}}^3$:

\begin{center}
\scalebox{1.1}{
$
\def\0{{\mathrm C}}
\def\A{\hat\alpha}
\def\C{\hat\gamma}
\def\B{\hat\beta}
\def\X{\cdot}
\begin{array}{|c@{\,}c@{\,}c@{\,}c|c@{\,}c@{\,}c@{\,}c|c@{\,}c@{\,}c@{\,}c|c@{\,}c@{\,}c@{\,}c|}
 \hline
\X & \X & \X & \C &  \X & \B & \X & \X &  \X & \X & \A & \X &  \0 & \X & \X & \X \\[-0.5EM]
\X & \X & \B & \X &  \C & \X & \X & \X &  \X & \X & \X & \0 &  \X & \A & \X & \X \\[-0.5EM]
\cellcolor{blue!15}\A & \X & \X & \X &  \X & \X & \0 & \X &  \X & \C & \X & \X &  \X & \X & \X & \B \\[-0.5EM]
\X & \0 & \X & \X &  \X & \X & \X & \A &  \B & \X & \X & \X &  \X & \X & \C & \X \\\hline

\B & \X & \X & \X &  \X & \X & \C & \X &  \X & \0 & \X & \X &  \X & \X & \X & \A \\[-0.5EM]
\X & \C & \X & \X &  \X & \X & \X & \cellcolor{red!20}\B &  \A & \X & \X & \X &  \X & \X & \0 & \X \\[-0.5EM]
\X & \X & \X & \0 &  \X & \A & \X & \X &  \X & \X & \B & \X &  \C & \X & \X & \X \\[-0.5EM]
\X & \X & \A & \X &  \0 & \X & \X & \X &  \X & \X & \X & \C &  \X & \B & \X & \X \\\hline

\X & \A & \X & \X &  \X & \X & \X & \0 &  \C & \X & \X & \X &  \X & \X & \B & \X \\[-0.5EM]
\0 & \X & \X & \X &  \X & \X & \A & \X &  \X & \B & \X & \X &  \X & \X & \X & \C \\[-0.5EM]
\X & \X & \C & \X &  \B & \X & \X & \X &  \X & \X & \X & \A &  \X & \0 & \X & \X \\[-0.5EM]
\X & \X & \X & \B &  \X & \cellcolor{red!20}\C & \X & \X &  \X & \X & \0 & \X &  \A & \X & \X & \X \\\hline

\X & \X & \0 & \X &  \A & \X & \X & \X &  \X & \X & \X & \B &  \X & \C & \X & \X \\[-0.5EM]
\X & \X & \X & \A &  \X & \0 & \X & \X &  \X & \X & \C & \X &  \B & \X & \X & \X \\[-0.5EM]
\X & \B & \X & \X &  \X & \X & \X & \C &  \0 & \X & \X & \X &  \X & \X & \A & \X \\[-0.5EM]
\C & \X & \X & \X &  \X & \X & \B & \X &  \X & \A & \X & \X &  \X & \X & \X & \0 \\\hline
\end{array}
$
}
\end{center}

We now assume, seeking a contradiction, that
$B_i 0 \cup B_i^1 1 \cup B_i^2 2 \cup B_i^3 3$, $i=0, \ldots, 15$,
form a partition of $H(5,4)$ into sixteen $1$-perfect codes.
In this case, $B_2^1$, $B_3^1$, $B_{\mathrm{C}}^1$
are mutually disjoint,
$B_2^2$, $B_3^2$, $B_{\mathrm{C}}^2$ are mutually disjoint,
and
$B_2^3$, $B_3^3$, $B_{\mathrm{C}}^3$ are mutually disjoint.
In the picture above, we see that
$B_2^\alpha$ intersects with
$B_3^{\tilde\beta}$
and with $B_3^{\tilde\gamma}$ (examples of intersection points
are marked by color background).
Hence, $\alpha = {\tilde\alpha}$.
Next, we see that
$B_3^{\tilde\alpha}$ intersects with $B_{\mathrm{C}}^{\hat\beta}$
and with $B_{\mathrm{C}}^{\hat\gamma}$.
Hence, $\alpha = \tilde\alpha = {\hat\alpha}$.
But $B_2^\alpha$ intersects with $B_{\mathrm{C}}^{\hat\alpha}$,
which implies $\alpha \ne \hat\alpha$, a contradiction.
As a result, we have a partition of
$H(4,4)$ into $(4,16,3)_4$ codes that cannot be lengthened
to a partition of $H(5,4)$ into $1$-perfect $(5,64,3)_4$ codes.

\begin{theorem}\label{th:4}
 For each $m=3,4, \ldots$, there is a $4$-ary
 $((4^{m}-4)/3, 4^{(4^{m}-4)/3 - m},3)_4$
 code which is not a shortened $1$-perfect code and
 is an element of a partition of the vertex set $H(4^{m}-4)/3,4)_4$ into
 codes with the same parameters.
\end{theorem}
\begin{proof}
 The proof is by induction on $m$.
 For $m=3$, the code is constructed as in
 Theorem~\ref{th:1}, utilizing the partition
 of $H(4,4)$ presented in Section~\ref{s:4444}.
 For $m>4$, the induction step is done
 using Theorem~\ref{th:1} with
 the non-lengthenable partition
 constructed at the previous step.
\end{proof}

\subsection{Other alphabets}\label{s:5555}

The construction above cannot produce non-lengthenable shortened-$1$-perfect-like codes for the ternary alphabet by the following reasons:
\begin{itemize}
\item \textit{there are only two inequivalent partitions of the vertex set of
$H(3,3)$ into $(3,3,3)_3$ codes, and each of them is lengthenable
to a partition of the vertex set of
$H(4,3)$ into $1$-perfect $(4,9,3)_3$ codes}.
\end{itemize}
We refer this fact as a computational result, while it can certainly be obtained by exhaustive search without computers.
The parameters
$(12,3^9,3)_3$ are not too big to handle such codes with computer,
but the number of inequivalent codes is too large and checking equivalence is too slow (from less than second to more than hour) to make an exhaustive search possible. On the other hand, the lengthenability
can be checked effectively: a ternary shortened-$1$-perfect like code $C$
can be lengthened to a $1$-perfect code if and only if the distance-$1$-and-$2$ graph of $C^{(2)}$ is bipartite (indeed, $C'$ and $C''$ are the parts of this graph if and only if the code $C0\cup C'1 \cup C''2$
is $1$-perfect).
So, one can try to find an example of a non-lengthenable code by some randomized approach. However, if such example does not exist in $H(12,3)$, this cannot be established with such an approach.

The approach of Section~\ref{s:4444} might be effective
for the alphabet of size $5$. We seek a partition of $H(5,5)$
into $25$ $(5,5^3,3)_5$ MDS codes that are not lengthenable to a partition of $H(6,5)$ into $25$ perfect codes. The exhaustive search
of all such partitions is impossible,
but some restrictions or theoretical ideas can reduce the amount of computation. Actually, we conjecture that the number of lengthenable partitions is relatively small in comparing with all partitions, for these parameters.

The size $8$ of the alphabet is the first value of $q$ for which
there are inequivalent $(q,q^{q-2},3)_q$ MDS codes,
with parameters of a shortened $1$-perfect code.
Such codes for $q\le 8$ are classified in~\cite{KKO:smallMDS}.
The fact that all of them are lengthenable is not mentioned there,
and we made a separate calculation to reproduce this result.
To do this, we shortened all $1$-perfect
$(9,8^{7},3)_8$ codes (they are listed in~\cite{KKO:smallMDS})
in all possible ways and counted the number of equivalence
classes of the resulting codes.
By curiosity reasons, we repeated this procedure for
shortened $1$-perfect $(8,8^{6},3)_8$ codes as well.
The results are as follows:
\begin{itemize}
\item \textit{All $4$ inequivalent $(8,8^{6},3)_8$ are shortened $1$-perfect}.
\item
\textit{Among the $8$ inequivalent $(7,8^{5},3)_8$ codes,
seven codes are doubly-shortened $1$-perfect,
while the one with $|\mathrm{Aut}(C)| = 86016$ (see~\cite[Table~I]{KKO:smallMDS}) is not}.
\end{itemize}

\section{Concluding remarks}\label{s:end}
In the first part of the paper, we proved
an upper bound on the size of a $\lambda$-fold packing
of $1$-balls in a non-binary Hamming space.

In the second part of the paper, we showed
that among non-linear
$(n=(q^m-q)/(q-1),q^{n-m-1},3)_q$ codes, in the case $q=4$,
for every $m\ge 3$ there are codes that are not shortened
$1$-perfect codes. We conjecture that the same is true for every prime power $q$ larger than $4$.
However, at this moment it is hard
to predict what happens in the case
$q=3$ and
in the case $m=2$, arbitrary $q$
(in the last case the codes are MDS). These two special
subcases of the considered question remain to be intriguing problems.

Partitions of the Hamming space into MDS codes (including the ones considered in Section~\ref{s:4444})
are a special subcase of more general family
of Sudoku-like arrays considered
in~\cite{HMST:2017:sudoku}.
The question which of such partitions
can be lengthened is of independent interest,
while finding non-lengthenable partitions
into $(q,q^{q-2},3)_q$ codes, $q>4$
would result in non-lengthenable
codes with parameters of shortened $1$-perfect
codes by Theorem~\ref{th:1}.


\providecommand\href[2]{#2} \providecommand\url[1]{\href{#1}{#1}}
  \def\DOI#1{{\small {DOI}:
  \href{http://dx.doi.org/#1}{#1}}}\def\DOIURL#1#2{{\small{DOI}:
  \href{http://dx.doi.org/#2}{#1}}}

\end{document}